\documentclass[11pt]{amsart}
\usepackage[margin=1in]{geometry}
\usepackage{amsfonts, amstext, amsmath, amssymb}
\usepackage{graphicx}
\pdfoutput=1

\newtheorem{lemma}{Lemma}
\newtheorem{theorem}[lemma]{Theorem}
\newtheorem{proposition}[lemma]{Proposition}

\begin{document}

\title{\textsf{On the construction of general equilibria in a competitive economy}}

\author[Hendtlass]{Matthew Hendtlass}
 \address{School of Mathematics and Statistics,
    University of Canterbury,
    Christchurch 8041,
    New Zealand}
\email{matthew.hendtlass@canterbury.ac.nz}

\author[Miheisi]{Nazar Miheisi}
 \address{King's College London, 
 Strand, 
 London WC2R 2LS, 
 United Kingdom}
\email{nazar.miheisi@kcl.ac.uk}

\begin{abstract}
\noindent \textsf{This paper gives a constructive treatment of McKenzie's theorem on the existence of general equilibria. While the full theorem does not admit a constructive proof, and hence does not admit a computational realisation, we show that if we strengthen the conditions on our preference relation---we require $\succ$ to be uniformly rotund in the sense of Bridges \cite{dsb_92}---then we can find `approximate equilibrium points,' points at which the collective profit may not be maximal, but can be made arbitrarily close to being maximal.}
\end{abstract}


\maketitle

\section{Introduction}

\noindent
In a series of papers \cite{dsb_82,dsb_89,dsb_92,dsb_94} Bridges gave a development of the theory of utility functions in constructive mathematics. In this paper we extend on these beginnings to a rigorously constructive development of mathematical economics. We consider the existence (explicit construction) of general equilibria in a competitive economy: Does there exist an algorithm to construct a competitive equilibrium (under economically reasonable conditions)? Since the construction of competitive equilibria seems to require Brouwer's fixed point theorem---which is not constructively valid---the answer would appear to be no. We satisfy ourselves with the construction of price vectors which are almost equilibria.

\bigskip
\noindent
We begin by giving a review of (Bishop's) constructive mathematics, the system in which this paper is written. Fundamentally, constructive mathematics differs from standard `classical' mathematics in the interpretation of the phrase `there exists'. Classically an object is asserted to exist if its failure to exist is contradictory; constructively we demand more: an object can be asserted to exist only if we can give a procedure which allows the construction of approximations, of arbitrary accuracy, in a finite time---dependent on the desired accuracy---with finite computational resources. Formally, constructive mathematics is essentially mathematics with intuitionistic logic and dependent choice\footnote{For an introduction to the practice of constructive mathematics see \cite{BB,BR,BV}; formally we view constructive mathematics as constructive \text{ZF} set theory \cite{AR} plus dependent choice, with intuitionistic logic}. Adding the \textbf{law of excluded middle} 
 \begin{quote}
  \textbf{LEM}: For each proposition $P$, either $P$ holds or $\neg P$ holds.
 \end{quote}
 \noindent
to constructive mathematics returns classical mathematics (with dependent choice, rather than the full axiom of choice). Working with intuitionistic logic ensures that proofs proceed in a manner which preserves computational content: a proof of $A\Rightarrow B$ converts a witness of $A$ into a witness of $B$. In particular, a constructive proof of $\exists_xP(x)$ embodies an algorithm for the construction of an object $x$, and an algorithm verifying that $P(x)$ holds. In this manner, constructive mathematics can be viewed as a high level programming language. One should note, however, that the constructive mathematician is interested only in what is within the realms of computability and not with computational efficiency. On the other hand, there have been (successful) efforts to extract efficient algorithms from constructive proofs \cite{HS}.

\bigskip
\noindent
We hope a (vague) example will illuminate the necessary mindset of the constructive mathematician, and the rewards for his or her hard work. The intermediate value theorem says that for every uniformly continuous function $f:[0,1]\rightarrow\mathbf{R}$ with $f(0)=-1$ and $f(1)=1$, there exists $x\in[0,1]$ such that $f(x)=0$. The classical argument is simple: define 
 $$
  x=\sup\{y\in[0,1]:f(y)<0\},
 $$
\noindent
then $f(x)=0$ by continuity. This proof, however, gives us no idea of the value of $x$ unless we can calculate this supremum, which we generally cannot. 

\bigskip
\noindent
Constructively, we must be more careful; indeed, it is in general not possible to construct such an $x$, but we can construct approximate zeros: for all $\varepsilon>0$ there exists $x\in[0,1]$ such that $|f(x)|\leqslant\varepsilon$. Here is such a construction. Fix $\varepsilon$ and let $n\in\mathbf{N}$ be such that for all $s,t\in[0,1]$, if $|s-t|\leqslant1/n$, then $|f(x)-f(y)|<\varepsilon/2$; we consider the set $S=\{0,1/n,\ldots,1\}$. We cannot generally decide whether $f(x)<0\vee f(x)=0\vee f(x)>0$ for each $x\in S$: we might check the first million digits of $f(x)$ before giving up and deciding it is probably zero, while the one millionth and first digit shows it to be positive (or negative)---to show $f(x)$ is zero requires checking infinitely many digits. By checking enough digits we can, however, decide $f(x)<0\vee |f(x)|<\varepsilon\vee f(x)>0$. 

\bigskip
\noindent
If for any element of $S$ we decide that $|f(x)|<\varepsilon$, then we are done. If, on the other hand, we have that $f(x)<0$ or $f(x)>0$ for each $x\in S$, then there must be some $m\in\{0,1,\ldots,n-1\}$ such that $f(m/n)<0$ and $f((m+1)/n)>0$. That $\max\{|f(m/n)|,|f((m+1)/n)|\}>\varepsilon$ would contradict our choice of $n$; whence both $|f(m/n)|$ and $|f((m+1)/n)|$ are less than $\varepsilon$, and in this case we are also done. This argument, which is essentially the one dimensional case of Scarf's algorithm for Brouwer's fixed point theorem \cite{Scarf1}, actually produces for us an approximate zero of the desired accuracy.

\bigskip
\noindent
We now turn our attention to competitive equilibria.

\bigskip
\noindent
We work in the economic model used by McKenzie \cite{M1,M2}; we have $N$ commodities, $n$ producers, and $m$ consumers. To each producer we associate a production set $Y_i\subset\mathbf{R}^N$; and to each consumer a consumption set $X_i\subset\mathbf{R}^N$ endowed with a preference relation $\succcurlyeq_i$. Further we assume that each consumer has no initial endowment, and we write $x\mathord{\sim}_ix^\prime$ for $x\succ_i x^\prime\vee x^\prime\succ_i x$ and $x \succ_i x^\prime$ for $x \succcurlyeq_i x^\prime\wedge x\not\mathord{\sim}_i x^\prime$. When the price $\mathbf{p}\in\mathbf{R}^N$ prevails, we define:
\begin{itemize}
 \item[$\blacktriangleright$] the \textbf{budget set} for consumer $i$ 
  $$
   \beta_i(\mathbf{p})=\left\{x_i\in X_i:\mathbf{p}\cdot\mathbf{x}_i\leqslant0\right\};
  $$
 \item[$\blacktriangleright$] the $\mathbf{p}$ \textbf{upper contour set} of consumer $i$
  $$ 
   C_i(\mathbf{p})=\left\{\mathbf{x}_i\in X_i:\mathbf{x}_i\succcurlyeq\mathbf{\xi}\mbox{ for all }\mathbf{\xi}\in\beta_i(\mathbf{p})\right\};
  $$
 \item[$\blacktriangleright$] consumer $i$'s \textbf{demand set}
  $$
   D_i(\mathbf{p})=\beta_i(\mathbf{p})\cap C_i(\mathbf{p}).
  $$
\end{itemize}
\noindent
In order to prove meaningful results we must restrict our attention to consumers who act in a suitably reasonable manner; the following restrictions on $\succcurlyeq$ give some measures of the reasonable consumer. Let $\succ$ be a preference relation on a subset $X$ of $\mathbf{R}^N$.
 \begin{itemize}
  \item[$\blacktriangleright$] $\succ$ is a \textbf{continuous preference relation} if 
  the graph 
   $$
    \{(x,x^\prime):x\succ x^\prime\}
   $$
  \noindent
  of $\succcurlyeq$ is open.
 \item[$\blacktriangleright$] $\succ$ is \textbf{strictly convex} if $X$ is convex and $tx+(1-t)x^\prime\succ x^\prime$ whenever $x\succcurlyeq x^\prime$, $x\neq x^\prime$, and $t\in(0,1)$.
 
 \item[$\blacktriangleright$] $X$ is \textbf{uniformly rotund} if for each $\varepsilon>0$ there exists $\delta>0$ such that for all $x,x^\prime\in X$, if $\Vert x-x^\prime\Vert\geqslant\varepsilon$, then 
 $$
  \left\{\frac{1}{2}\left(x+x^\prime\right)+z:z\in B(0,\delta)\right\}\subset X,
 $$
\noindent
where $B(x,r)$ is the open ball of radius $r$ centred on $x$. The preference relation $\succ$ is uniformly rotund if $X$ is uniformly rotund and for each $\varepsilon>0$ there exists $\delta>0$ such that if $\Vert x-x^\prime\Vert\geqslant\varepsilon$ ($x,x^\prime\in X$), then for each $z\in B(0,\delta)$ either $\frac{1}{2}\left(x+x^\prime\right)+z\succ x$ or $\frac{1}{2}\left(x+x^\prime\right)+z\succ x^\prime$.
 \end{itemize}
\noindent
Bridges showed in \cite{dsb_93} that a uniformly rotund preference relation is strictly convex. 

\bigskip
\noindent
A set $S$ is said to be \textbf{inhabited} if there exists $x$ such that $x\in S$. An inhabited subset $S$ of a metric space $X$ is \textbf{located} if for each $x\in X$ the \textbf{distance} 
 \begin{equation*}
  \rho \left( x,S\right) =\inf \left\{ \rho (x,s):s\in S\right\}
 \end{equation*}
\noindent 
from $x$ to $S$ exists. An $\varepsilon$\textbf{-approximation} to $S$ is a subset $T$ of $S$ such
that for each $s\in S$, there exists $t\in T$ such that $\rho
(s,t)<\varepsilon $. We say that $S$ is \textbf{totally bounded} if for each $%
\varepsilon >0$ there exists a finitely enumerable $\varepsilon $-approximation to $S$. If each bounded subset of $S$ is contained in a totally bounded subset of $S$, then $S$ is said to be \textbf{locally totally bounded}. A metric space $X$ is \textbf{compact} if it is both complete and totally bounded.

\bigskip
\noindent
In \cite{dsb_92} Bridges gave a constructive proof of the following theorem (see also \cite{mh2}).

\begin{theorem}
\label{demand}
Let $\succ$ be a uniformly rotund preference relation on a compact subset $X$ of $\mathbf{R}^N$, let $P$ be a bounded set of nonzero vectors in $\mathbf{R}^N$ such that for each $\mathbf{p}\in P$,
 \begin{itemize}
  \item[(i)] $\beta(\mathbf{p})$ is inhabited;
  \item[(ii)] there exists $\mathbf{\xi}\in X$ such that $\mathbf{\xi}\succ\mathbf{x}$ for all $\mathbf{x}\in\beta(\mathbf{p})$.
 \end{itemize}
\noindent
Then $D(\mathbf{p})$ consists of a unique point $F(\mathbf{p})$, the function $F$ is uniformly continuous, and $\mathbf{p}\cdot F(\mathbf{p})=0$.
\end{theorem}

\noindent
The function $F$ constructed in Theorem \ref{demand} is called the \textbf{demand function} of $(X,\succcurlyeq)$. Classically, we only require $\succ$ to be strictly convex and continuous in order to prove the existence of a demand function \cite{Tak}.

\bigskip
\noindent
A \textbf{competitive equilibrium} of an economy consists of a price vector $\mathbf{p}\in\mathbf{R}^N$, points $\mathbf{\xi}_1,\ldots,\mathbf{\xi}_i\in\mathbf{R}^N$, and a vector $\mathbf{\eta}$ in the \textbf{aggregate production set}
 $$
  Y=Y_1+\cdots+Y_n,
 $$
\noindent
satisfying
 \begin{itemize}
  \item[\textbf{E1}] $\mathbf{\xi}_i\in D_i(\mathbf{p})$ for each $1\leqslant i\leqslant m$.
  \item[\textbf{E2}] $\mathbf{p}\cdot\mathbf{y}\leqslant\mathbf{p}\cdot\mathbf{\eta}=0$ for all $\textbf{y}\in Y$.
  \item[\textbf{E3}] $\sum_{i=1}^m\mathbf{\xi}_i=\mathbf{\eta}$.
 \end{itemize}
\noindent
An economy is said to have \textbf{approximate competitive equilibria} if for all $\varepsilon>0$ there exist a price vector $\mathbf{p}\in\mathbf{R}^N$, points $\mathbf{\xi}_1,\ldots,\mathbf{\xi}_i\in\mathbf{R}^N$, and a vector $\mathbf{\eta}$ satisfying \textbf{E1},\textbf{E3}, and 
 \begin{itemize}
  \item[\textbf{AE}] $\mathbf{p}\cdot\mathbf{\eta}>-\varepsilon$ and $p\cdot y\leqslant0$ for each $y\in Y$.
 \end{itemize}
\noindent
Alternatively, an economy has approximate equilibria if
 $$
  \inf\{\mathbf{p}\cdot(\xi_1,\ldots,\xi_m):\xi_i\in D_i(\mathbf{p})\mbox{ for each }1\leqslant i\leqslant m\}=0.
 $$

\bigskip
\noindent
In an approximate equilibrium each consumer maximises his utility while, in contrast, each firm only approximately maximises its profits. Why not demand that profit is maximised, and allow consumers utility to deviate from the optimal? There are a few reasons why the above definition gives the appropriate notion of approximate equilibrium in the context of constructive mathematics. Our task is to construct a price vector $\mathbf{p}$ satisfying our equilibrium condition, for once this is done the $\xi_i$ are given by Theorem \ref{demand}. Thus it is \textbf{E2} which requires the construction of a fixed point, and hence is not possible constructively. So we are forced, at an approximate equilibrium, to allow firms to make losses; however, these losses can be made arbitrarily small, and a loss of one millionth of a cent, for instance, is no loss at all. Anyway, in a real economy things are sold in units---we deal with discrete set of consumer bundles---and we are not in general able to maximise profit, or consumer utility.


\bigskip
\noindent
Let $F_i$ denote the demand function on $(X_i,\succcurlyeq_i)$. A subset $Y$ of a normed space is said to be a \textbf{convex cone} if $\lambda y\in Y$ and $y+y^\prime\in Y$ whenever $y,y^\prime\in Y$ and $\lambda\geqslant0$. The \textbf{convex conic closure} $\mathrm{cone}(Y)$ of $Y$ is the smallest convex cone containing $Y$; that is,
 $$
  \mathrm{cone}(S)=\{r(tx+(1-t)y):r>0,t\in[0,1],x,y\in S\}.
 $$
\noindent
We use $S^\circ$ to denote the interior of a subset $S$ of a metric space.

\bigskip
\noindent
We can now state McKenzie's theorem on the existence of competitive equilibria.

\begin{quote}
 \textbf{McKenzie's Theorem}\newline\noindent
 \emph{Suppose that} 
  \begin{itemize}
   \item[\emph{(i)}] \emph{each $X_i$ is compact and convex;}
   \item[\emph{(ii)}] \emph{each $\succcurlyeq_i$ is continuous and strictly convex;}
   \item[\emph{(iii)}] \emph{$\left(X_i\cap Y\right)^\circ$ is nonempty for each $i$;}
   \item[\emph{(iv)}] \emph{$Y$ is a closed convex cone;}
   \item[\emph{(v)}] \emph{$Y\cap\left\{\left(x_1,\ldots,x_N\right):x_i\geqslant0\mbox{ for each }i\right\}=\{0\}$; and}
   \item[\emph{(vi)}] \emph{for each $\mathbf{p}\in\mathbf{R}^N$ and each $i$, if $\sum_{i=1}^m F_i(\mathbf{p})\in Y$, then there exists $\textbf{x}_i\in X_i$ such that $\mathbf{x}_i\succ_i F_i(\mathbf{p})$.}
  \end{itemize}
 \noindent
 \emph{Then there exists a competitive equilibrium.}
\end{quote}

\bigskip
\noindent
The standard proofs of McKenzie's theorem on the existence of general equilibria all contain seemingly necessary applications of Brouwer's fixed point theorem (often in the guise of Kakutani's fixed point theorem). Since the construction of exact fixed points is not possible---even with strong assumptions on the function, like the uniqueness of any fixed point \cite{Orehkov, V}---it seems unlikely that a constructive proof of the existence of exact competitive equilibria is possible under any economically reasonable assumptions.

\bigskip
\noindent
It may seem that given Bridges constructive proof of Theorem \ref{demand} and Scarf's algorithm for finding approximate fixed points \cite{Scarf1}, that all the hard work for giving a computational version of McKenzie's theorem has already been done. This is not the case: much care and attention must be given to the construction of the, family, of set valued mappings to which we will apply Kakutani's fixed point theorem.

\section{Constructing competitive equilibria}

\noindent
Our constructive version of McKenzie's theorem is the following.

\begin{theorem}
 \label{cMcKenzie}
 Suppose that 
  \begin{itemize}
   \item[(i)] each $X_i$ is compact and convex;
   \item[(ii)] each $\succcurlyeq_i$ is continuous and uniformly rotund;
   \item[(iii)] $\left(X_i\cap Y\right)^\circ$ is inhabited for each $i$;
   \item[(iv)] $Y$ is a located closed convex cone;
   \item[(v)] $Y\cap\left\{\left(x_1,\ldots,x_N\right):x_i\geqslant0\mbox{ for each }i\right\}=\{0\}$; and
   \item[(vi)] for each $\mathbf{p}\in\mathbf{R}^N$ and each $i$, if $\sum_{i=1}^m F_i(\mathbf{p})\in Y$, then there exists $\textbf{x}_i\in X_i$ such that $\mathbf{x}_i\succ_i F_i(\mathbf{p})$.
  \end{itemize}
 \noindent
 Then there are approximate competitive equilibria.
\end{theorem} 

\noindent
Our proof follows the standard classical proof via Kakutani's fixed point theorem (see \cite{M3}) as closely as possible; typical of constructive mathematics, it has a distinctly geometric character. The \textbf{polar} of a subset $S$ of $\mathbf{R}^N$ is the set
 $$
  S^{\mathrm{pol}}=\left\{\mathbf{p}\in\mathbf{R}^N:\mathbf{p}\cdot\mathbf{x}\leqslant0\mbox{ for all }\mathbf{x}\in S\right\}.
 $$
\noindent
It follows directly from the definition that two sets $S,T\subset \mathbf{R}^N$ have the same polar if and only if they have the same convex conic closure. Classically, a little further work shows that 
\begin{quote}
(*) the polar of the polar of a set is equal to its convex conic closure,
\end{quote}
\noindent
but this is not the case in our intuitionistic setting as the following `Brouwerian counterexample\footnote{A Brouwerian counterexample is a weak counterexample: it is not an example contradicting a proposition, but an example showing a proposition to imply a principle which is unacceptable in constructive mathematics. Generally these can be considered as independence proofs.}' shows.

\bigskip
\noindent
Given a proposition $P$, let $S$ be the set
 $$
  S = \{(0,1)\}\cup\{x:x=(1/2,0)\wedge P\}\cup\{x:x=(1,0)\wedge\neg P\}.
 $$
\noindent
Then $(1,1)\in (S^\mathrm{pol})^\mathrm{pol}$. Suppose that $(1,1)\in \mathrm{cone}(S)$; that is, suppose there exist $r>0$, $t\in[0,1]$, and $x,y\in S$ such that 
 $$ 
  (1, 1) = r(tx + (1 - t)y).
 $$
\noindent
Without loss of generality, we may suppose $x=(0,1)$. Then either $r>2$, in which case $y$ must be $(1/2,0)$ and so $P$ holds, or $r<3$ and, similarly, $\neg P$ must hold. Hence (*) implies the law of excluded middle.

\bigskip
\noindent
The above counterexample is rather contrived and seems to have little to do with real mathematics or economics, but seems merely to indicate how one would show (*) to be independent of some formalisation of constructive mathematics. It is, however, relevant: since our framework encapsulates what is computable in a strict, though ill-defined, way, this example shows that we cannot compute the information implicit in ``$x\in\mathrm{cone}(S)$''---that there exists $r>0,t\in[0,1],y,y^\prime\in S$ such that $x=r(ty+(1-t)y^\prime)$---given only the information that for all $z$, if $z\cdot p\leqslant0$ for all $p\in S$, then $x\cdot z=0$. Succinctly, belonging to the conic closure of a set gives more computational information than belonging to the polar of the polar of that set, and when it comes to computational information we cannot get something for nothing! 

\bigskip
\noindent
The above failure of (*) results from us having a poor handle on the set $S$. The sets we deal with in practice are generally more well behaved, and (*) can be proved, constructively, for a large class of sets. The following result meets our needs.

\begin{proposition}
 Let $S$ be a located closed convex cone in $\mathbf{R}^N$. Then the polar of the polar of $S$ equals $S$.
\end{proposition}

\begin{proof}
By definition $x\in (S^\mathrm{pol})^\mathrm{pol}$ if and only if
 $$
  \bigcap_{s\in S}\{z\in\mathbf{R}^N: z\cdot s\leqslant0\}\ =\ S^\mathrm{pol}\ \subset\ \{x\}^\mathrm{pol}.
 $$
\noindent
The assumption that $\rho(x, S)>0$ would contradict the above equation: let $y$ be the closest point to $x$ in $S$, this exists by Theorem 6 of \cite{BRS}. Since $S$ is a closed convex cone, $y-x\in S^\mathrm{pol}$, but $x-y\notin\{x\}^\mathrm{pol}$. Hence $\rho(x,S)=0$, and, since $S$ is closed, $x\in S$. The converse is straightforward.
\end{proof}

\bigskip
\noindent
For each $i$ we fix $\overline{\mathbf{\xi}}_i\in \left(X_i\cap Y\right)^\circ$ and let $\overline{\mathbf{\xi}}=\sum_{i=1}^m\overline{\mathbf{\xi}}_i$; without loss of generality, each term of $\overline{\mathbf{\xi}}$ is nonzero. The proof of Theorem \ref{cMcKenzie} proceeds by an application of a constructively valid version of Kakutani's fixed point theorem to the set
 $$
  P=\left\{\mathbf{p}\in Y^{\mathrm{pol}}:\mathbf{p}\cdot\overline{\mathbf{\xi}}=-1\right\}
 $$
\noindent
of normalised price vectors. First, however, we require a number of lemmas. For the remainder of the paper we assume that the hypothesis of Theorem \ref{cMcKenzie} hold. We denote by $B(x,r)$ and $\overline{B}(x,r)$ the open and closed balls, respectively, centred on $x$ with radius $r$.

\begin{lemma}
\label{1}
If $\mathbf{y}\in Y^\circ$, then $\mathbf{p}\cdot\mathbf{y}<0$ for all nonzero $\mathbf{p}\in Y^{\mathrm{pol}}$. Moreover, $\sup\{\mathbf{p}\cdot\mathbf{y}:\mathbf{p}\in Y^{\mathrm{pol}}, \Vert p\Vert=1\}<0$.
\end{lemma}

\begin{proof}
Let $\mathbf{p}$ be a nonzero element of $Y^{\mathrm{pol}}$; pick $1\leqslant i\leqslant N$ such that $p_i\neq0$, and fix $r>0$ such that $\overline{B}(\mathbf{y},r)\subset Y$. Then
 $$
  \mathbf{y}^\prime\equiv\mathbf{y}+(\mathrm{sign}(p_i)r)\mathbf{e}_i\in Y,
 $$
\noindent
where $\mathbf{e}_i$ is the $i$th basis vector. Hence
 $$
  \mathbf{p}\cdot\mathbf{y}<\mathbf{p}\cdot\mathbf{y}+\vert rp_i\vert=\mathbf{p}\cdot\mathbf{y}^\prime\leqslant0.
 $$
\noindent
If  $\Vert p\Vert=1$, then we may suppose that $|p_i|>1/2\sqrt{N}$; thus $\mathbf{p}\cdot\mathbf{y}<-|rp_i|<-r/2\sqrt{N}$.
\end{proof}

\bigskip
\noindent
Let $S$ be a subset of a metric space $X$. The \textbf{complement} of $S$ is 
 $$
  \mathord{\sim}S=\{x\in X:x\neq s\mbox{ for all }s\in S\}.
 $$
\noindent
If $S$ is located, then the \textbf{apartness complement} of $S$ is the set
 $$ 
  -S=\{x\in X:\rho(x,S)>0\}.
 $$ 

\begin{lemma}
\label{1'}
For each $i$ the demand function $F_i$ for $X_i$ maps into $\mathord{\sim}(Y^\circ)$.
\end{lemma}

\begin{proof}
Suppose that $F(\mathbf{p})\in Y^\circ$. Then, by Lemma \ref{1}, $\mathbf{p}\cdot F(\mathbf{p})<0$, which contradicts Theorem \ref{demand}. 
\end{proof}

\bigskip
\noindent
Let $C$ be a located convex subset of a Banach space $X$. Then for each $\xi\in C^\circ$ and each $z\in-C$ there exists a unique point $h(\xi,z)$ in the intersection of the \textbf{interval}
 $$
  [\xi,z]=\left\{t\xi+(1-t)z:t\in[0,1]\right\}
 $$
\noindent
and the boundary $\partial C$ of $C$; moreover, the mapping $(\xi,z)\mapsto h(\xi,z)$---the \textbf{boundary crossing map} of $C$---is pointwise continuous on $C^\circ\times-C$ \cite[Proposition 5.1.5]{BV}. The next lemma shows that for a fixed $\xi\in C^\circ$, this mapping is uniformly continuous.

\begin{lemma}
\label{2}
Let $X$ be a bounded convex subset of $\mathbf{R}^N$ and let $\mathbf{\xi}\in X^\circ$. Then the function $h:\mathbf{R}^N\rightarrow \overline{X}$ which fixes each point of $\overline{X}$ and sends $y\in\mathord{\sim}X$ to the unique intersection point of $[\mathbf{\xi},y]$ and $\partial X$ is uniformly continuous.
\end{lemma}

\begin{proof}
Without loss of generality we suppose $\mathbf{\xi}=0$. Let $N>0$ be such that $X\subset B(0,N)$ and let $r>0$ be such that $B(0,r)\subset X$. Since the function mapping a point $x\neq0$ to the unique intersection point of 
 $$
  \mathbf{R}x=\{rx:r\in\mathbf{R}\}
 $$
\noindent
and $\partial B(0,N)$ is uniformly continuous on $-B(0,r/2)$, it suffices to show that $h$ is uniformly continuous on $\partial B(0,N)$.

\bigskip
\noindent
Given $\delta>0$, set 
 \begin{eqnarray*}
  \theta & = & \mathrm{cos}^{-1}(1-(\delta^2/2N^2))\\
  \beta & = & \mathrm{cos}^{-1}(\delta/2N), \mbox{ and}\\
  \alpha & = & \mathrm{sin}^{-1}(r/N).
 \end{eqnarray*}
Define
 $$
  \varphi(\delta)=\frac{\delta|\mathrm{sin}(\beta)|}{|\mathrm{sin}(\alpha + \theta)|}.
 $$
\noindent
The function $\varphi$ is constructed as a `worst case scenario' given that $X$ contains $B(0,r)$ and is strictly contained in $B(0,N)$; see the following diagram.

\begin{center}
\includegraphics[width=8cm]{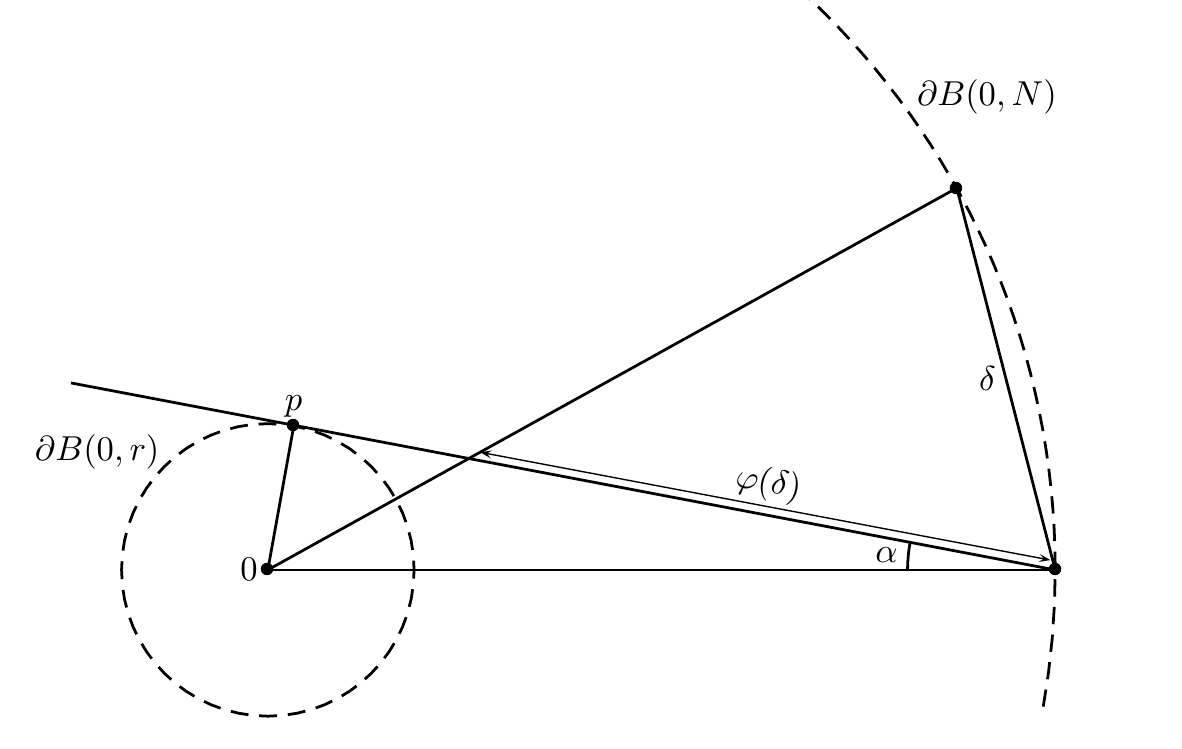}
\end{center}

\noindent
Fix $a,b\in \partial B(0,N)$ with $0<\|a-b\|<\delta$, and let $x\in[0,a]\cap X$ and $y\in[0,b]\cap X$ such that $\|x-y\|>\varphi(\delta)$; without loss of generality, $\|x\|<\|y\|$. It suffices to show that it cannot occur that both $x,y\in\partial X$, for then the assumption that $\|h(x)-h(y)\|>\varphi(\delta)$ leads to a contradiction. By the construction of $\varphi$, the unique line passing through $x$ and $y$ must intersect $B(0,r)$. It follows that 
$$
 x\in\left(\mathrm{conhull}\left(B(0,r)\cup\{y\}\right)\right)^\circ\subset X^\circ,
$$
\noindent
where $\mathrm{conhull}(S)$ is the convex hull of $S$. Hence if $\|a-b\|<\delta$, then $\|h(a)-h(b)\|\leqslant\varphi(\delta)$.

\bigskip
\noindent
It only remains to show that for each $\varepsilon>0$ we can find a $\delta>0$ such that $\varphi(\delta)<\varepsilon$. From elementary calculations we have that

\begin{align*}
 \varphi(\delta)&=\frac{\delta\sqrt{4N^2-\delta^2}}{2r(1-(\delta^2/2N^2))+2\delta\sqrt{(1-(r^2/N^2))(1-(\delta^2/2N^2))}} \\
&\leq \frac{\delta\sqrt{4N^2-\delta^2}}{2r(1-(r^2/2N^2))+2\delta(1-(r^2/N^2))} \ \ \longrightarrow\ \ 0
\end{align*}

\noindent
 as $\delta\to 0$.
\end{proof}

\begin{lemma}
\label{3}
Let $X,Y$ be convex subsets of a normed space such that $X,Y$ are both totally bounded, and $(X\cap Y)^\circ$ is inhabited. Then $X\cap Y$ is totally bounded.
\end{lemma}

\begin{proof}
Let $\xi\in (X\cap Y)^\circ$ and let $R>0$; without loss of generality $\xi\in B(0,R)$. Let $Y^\prime=Y\cap B(0,R)$ and let $h$ be the uniformly continuous function which fixes $X$ and maps each point $y$ in $-X$ to the unique point in $[\xi,z]\cap\partial X$. Fix $\varepsilon>0$ and let $\delta\in(0,\varepsilon/4)$ be such that if $\Vert y-y^\prime\Vert<\delta$, then $\Vert h(y)-h(y^\prime)\Vert<\varepsilon/4$. Let $\left\{y_1,\ldots,y_k\right\}$ be a $\delta/2$-approximation of $Y$ and partition $\{1,\ldots,k\}$ into disjoint sets $P,Q$ such that
 \begin{eqnarray*}
  i\in P & \Rightarrow & \rho(y_i,X)<\delta;\\
  i\in Q & \Rightarrow & \rho(y_i,X)>\delta/2.
 \end{eqnarray*}
\noindent
If $i\in P$, then there exists $x\in X$ such that $\rho(x,y_i)<\delta$. Then
 $$
  \Vert y_i-h(y_i)\Vert\leqslant\Vert y_i-x_i\Vert+\Vert x_i-h(y_i)\Vert<  \varepsilon/4+\varepsilon/4=\varepsilon/2
 $$
\noindent
and, since $Y$ is convex, $h(y_i)\in X\cap Y$. The set
 $$
  S=\left\{h\left(y_i\right):i\in P\right\}
 $$
\noindent
is an $\varepsilon$-approximation of $X\cap Y\cap B(0,R)=X\cap Y^\prime$: fix $z\in X\cap Y$ and pick $1\leqslant i\leqslant k$ such that 
 $$
  \Vert z-y_i\Vert<\delta/2.
 $$
\noindent
Then $i\in P$, so $h(y_i)\in S$ and 
 \begin{eqnarray*}
  \Vert z-h(y_i)\Vert & \leqslant & \Vert z-y_i\Vert+\Vert y_i-h(y_i)\Vert\\
   & < & \delta/2+\varepsilon/2 \ < \ \varepsilon.
 \end{eqnarray*}
\end{proof}

\begin{lemma}
\label{4}
$P$ is compact and convex.
\end{lemma}

\begin{proof}
It is straightforward to show that $P$ is closed and convex; it just remains to show that $P$ is totally bounded. By the bilinearity of the mapping $(\mathbf{p},\mathbf{x})\mapsto\mathbf{p}\cdot\mathbf{x}$, both $Y^{\mathrm{pol}}$ and
 $$
  \{\mathbf{p}\in\mathbf{R}^N:\mathbf{p}\cdot\overline{\xi}=-1\}
 $$
\noindent
are locally totally bounded. Since $P$ is the intersection of these two sets, $P$ is locally totally bounded by Lemma \ref{3}. It remains to show that $P$ is bounded: by Lemma \ref{1} 
 $$
  M=\sup\{\mathbf{p}\cdot\mathbf{y}:\mathbf{p}\in Y^{\mathrm{pol}}, \Vert p\Vert=1\}<0.
 $$
\noindent
Suppose that there exists $\mathbf{p}\in P$ such that $\Vert \mathbf{p}\Vert>-1/M$. Then $\mathbf{p}/\Vert\mathbf{p}\Vert\in Y^\mathrm{pol}$ and $(\mathbf{p}/\Vert\mathbf{p}\Vert)\cdot\overline{\xi}=-1/\Vert\mathbf{p}\Vert>M$---a contradiction.
\end{proof}

\begin{lemma}
\label{5}
For each $\mathbf{y}\in Y$ and each $r>0$ there exists $\mathbf{p}\in P$ such that $\mathbf{p}\cdot\mathbf{y}>-r$.
\end{lemma}

\begin{proof}
Fix $\mathbf{y}\in\partial Y$. Suppose that
 $$
  \sup\{\mathbf{p}\cdot\mathbf{y}:\mathbf{p}\in P\}<0;
 $$
\noindent
this supremum exists since $P$ is totally bounded and $(\mathbf{p},\mathbf{x})\mapsto\mathbf{p}\cdot\mathbf{x}$ is uniformly continuous. Then there exists $\mathbf{z}\in -Y$ such that $\mathbf{p}\cdot\mathbf{z}<0$ for all $\mathbf{p}\in P$. But
 $$
  \mathbf{z}\in P^{\mathrm{pol}}=(Y^{\mathrm{pol}})^{\mathrm{pol}}=Y.
 $$
\noindent
This contradiction ensures that $\sup\{\mathbf{p}\cdot\mathbf{y}:\mathbf{p}\in P\}=0$, from which the result follows.
\end{proof}

\bigskip
\noindent
For the proof of Theorem \ref{cMcKenzie}, we need a constructive version of Kakutani's fixed point theorem. Since Kakutani's fixed point theorem is a generalisation of Brouwer's fixed point theorem, it is not valid in constructive mathematics. Before introducing the constructively valid form of Kakutani's fixed point theorem we shall use, we require several definitions. We use $\mathcal{P}^*(X)$ to denote the class of nonempty located subsets of a set $X$. A set valued mapping on $X$ is a function from $X$ into $\mathcal{P}^*(X)$; the graph of a set valued mapping $\Phi$ from $X$ into $Y$ is the subset
 $$
  G(\Phi)=\bigcup_{x\in X}\{x\}\times \Phi(x)
 $$
\noindent
of $X\times Y$. A set valued mapping $\Phi$ on a metric space $X$ is said to be \textbf{weakly approximable} if for each $\varepsilon>0$, there exists
 \begin{itemize}
  \item[$\blacktriangleright$] a positive real number $\delta$,
  \item[$\blacktriangleright$] a $\delta/2$-approximation $S$ of $X$, and
  \item[$\blacktriangleright$] a function $\Phi^\prime$ from $S$ into $\mathcal{P}^*(X)$ with $G(\Phi^\prime)\subset G(\Phi)$,
 \end{itemize}
\noindent
such that if $x,x^\prime\in S$, $\Vert x-x^\prime\Vert<\delta$, $u\in \Phi^\prime(x)$, $u^\prime\in \Phi^\prime(x^\prime)$, and $t\in[0,1]$, then 
 $$
  \rho\left(\left(z_t,u_t\right),G(\Phi)\right)<\varepsilon,
 $$
\noindent
where $z_t=ty+(1-t)y^\prime$ and $u_t=tu+(1-t)u^\prime$. If $\Phi^\prime$ can be chosen independently of $\varepsilon$, in which case $S$ is a dense subset of $X$, then $\Phi$ is said to be \textbf{weakly approximable with respect to} $\Phi^\prime$. The following is Theorem 10 of \cite{mh}.

\begin{theorem}
\label{Kak}
 Let $X$ be a compact convex subset of $\mathbf{R}^N$ and let $\{\Phi_r:r\in\mathbf{R}^+\}$ be a family of set valued mappings on $X$ such that
  \begin{itemize}
   \item[(i)] for all $r,r^\prime\in\mathbf{R}$, if $r<r^\prime$, then there exists $\delta>0$ such that 
    $$
     \left\{x\in\mathbf{R}^N: \exists_{z\in G\left(\Phi_r\right)}\rho(x,z)<\delta\right\}\subset G\left(\Phi_{r^\prime}\right);
    $$
   \item[(ii)] $\Phi_r$ is weakly approximable for each $r\in\mathbf{R}^+$.
  \end{itemize}
 \noindent
 Then for each $r\in\mathbf{R}^+$, there exists $x\in X$ such that $x\in\Phi_r(x)$.
 \end{theorem}

\noindent
The proof of the following simple lemma is left to the reader.

\begin{lemma}
\label{6}
The composition of a weakly approximable mapping with a uniformly continuous function is weakly approximable.
\end{lemma}

\begin{lemma}
\label{7}
For a fixed $r>0$ and for each $\mathbf{z}\in\partial Y$, define 
 $$
  g_r(\mathbf{z})=\left\{\mathbf{p}\in P:\mathbf{p}\cdot\mathbf{z}>-r\right\}.
 $$
\noindent
Then $g_r(\mathbf{z})$ is inhabited and located for each $r>0$, and $g_r$ is weakly approximable. 
\end{lemma}

\begin{proof}
That $g_r(\mathbf{z})$ is inhabited for each $\mathbf{z}$ follows from Lemma \ref{5}. Fix $\varepsilon>0$ and let $\delta>0$ be such that for all $\mathbf{z},\mathbf{z}^\prime\in\mathbf{R}^N$, if $\Vert \mathbf{z}-\mathbf{z}^\prime\Vert<\delta$, then $\Vert \mathbf{p}\cdot \mathbf{z}-\mathbf{p}\cdot \mathbf{z}^\prime\Vert<r/2$ for all $\mathbf{p}\in P$---such a $\delta$ exists since the mapping $(\mathbf{p},\mathbf{x})\mapsto\mathbf{p}\cdot\mathbf{x}$ is uniformly continuous and $P$ is totally bounded. Let $\mathbf{z,z}^\prime\in\mathbf{R}^N$ be such that $\Vert \mathbf{z}-\mathbf{z}^\prime\Vert<\delta$ and let $\mathbf{p}\in g_{r/2}(\mathbf{z})$ and $\mathbf{p}^\prime\in g_{r/2}(\mathbf{z}^\prime)$. For each $t\in[0,1]$, let $\mathbf{p}_t=t\mathbf{p}+(1-t)\mathbf{p}^\prime$ and $\mathbf{z}_t=t\mathbf{z}+(1-t)\mathbf{z}^\prime$. Then for all $t\in[0,1]$ we have
 \begin{eqnarray*}
  \mathbf{p}_t\cdot\mathbf{z}_t & = & (t\mathbf{p}+(1-t)\mathbf{p}^\prime)\cdot(t\mathbf{z}+(1-t)\mathbf{z}^\prime)\\
   & = & t^2\mathbf{p}\cdot\mathbf{z}+t(1-t)(\mathbf{p}\cdot\mathbf{z}^\prime+\mathbf{p}^\prime\cdot\mathbf{z})+(1-t)^2\mathbf{p}^\prime\cdot\mathbf{z}^\prime\\
  & > & -t^2r/2-2t(1-t)r-(1-t)^2r/2\ =\ -r.
 \end{eqnarray*}
\noindent
Hence $g_r$ is weakly approximable with respect to $g_{r/2}$. That $g_r(\mathbf{z})$ is located for each $\mathbf{z}\in \partial Y$ follows from Theorem (4.9) on page 98 of \cite{BB}, and the uniform continuity of the mapping $\mathbf{p}\mapsto\mathbf{p}\cdot\mathbf{y}$ on $P$.
\end{proof}

\bigskip
\noindent
We now have the \textbf{proof of Theorem \ref{cMcKenzie}:}

\bigskip

\begin{proof}
Let $F_i$ be the demand function for the $i$th consumer and let
 $$
  F=\sum_{i=1}^mF_i.
 $$
\noindent
Fix $\varepsilon>0$ and let $\delta>0$ be such that for all $\mathbf{p}\in P$, if $\Vert \mathbf{x}-\mathbf{x}^\prime\Vert<\delta$, then $\Vert \mathbf{p}\cdot\mathbf{x}-\mathbf{p}\cdot\mathbf{x}^\prime\Vert<\varepsilon/2$. Set 
 $$
  m=\min\left\{\frac{\varepsilon}{2},
  \frac{\delta}{\sup\left\{\left\Vert\overline{\mathbf{\xi}}-\mathbf{\eta}\right\Vert:\eta\in F(P)\right\}}\right\}.
 $$
\noindent
For each $r>0$, define a set valued mapping $\Phi_r$ on $P$ by
 $$
  \Phi_r=g_{r}\circ h\circ F,
 $$
\noindent
where $h,g_r$ are as in Lemma \ref{2} (for $X=\sum_{i=1}^mX_i$) and Lemma \ref{7} respectively; $\Phi_r$ is well defined by Lemma \ref{1'}. By Lemmas \ref{2},\ref{4},\ref{6},\ref{7} and Theorem \ref{demand}, $\Phi_r$ is approximable for each $r>0$ and $P$ is compact and convex. Using Theorem \ref{Kak}, construct $\mathbf{p}\in P$ such that 
 $$
  \mathbf{p}\in g_m\circ h\circ F(\mathbf{p}).
 $$
\noindent
Set $\mathbf{\xi}_i=F_i(\mathbf{p})$ for each $i$, and set $\mathbf{\eta}=F(\mathbf{p})$. Then, by definition, the $\mathbf{\xi}_i$ satisfy condition \textbf{E1}, and $\mathbf{\eta}$ satisfies \textbf{E3}. Pick $t\in[0,1)$ such that 
 $$
  \mathbf{\zeta}\equiv h(F(\mathbf{p}))=t\overline{\mathbf{\xi}}+(1-t)\mathbf{\eta}.
 $$
\noindent
Since $\mathbf{p}\in g_m(\mathbf{\zeta})$ and $\mathbf{\eta}\in Y$,
 $$ 
  -m<\mathbf{p}\cdot\mathbf{\zeta}=t\mathbf{p}\cdot\overline{\mathbf{\xi}}+(1-t)\mathbf{p}\cdot\mathbf{\eta}\leqslant -t,
 $$
\noindent
so $t<m$; whence $\Vert \mathbf{\zeta}-\mathbf{\eta}\Vert<\delta$. By our choice of $\delta$, it follows that $\Vert \mathbf{p}\cdot\mathbf{\eta}-\mathbf{p}\cdot\mathbf{\zeta}\Vert<\varepsilon/2$. Thus $\mathbf{p}\cdot\mathbf{\eta}>-\varepsilon$, so \textbf{AE} is satisfied. 
\end{proof}

\bigskip
\noindent
With the help of \textbf{weak K\"{o}nig's lemma}
 \begin{quote}
  \textbf{WKL} Every infinite binary tree has an infinite path.
 \end{quote}
\noindent
we can recover the existence of an exact competitive equilibrium in the conclusion of Theorem \ref{cMcKenzie}: repeatedly apply Theorem \ref{cMcKenzie} to construct sequences $\left(\mathbf{p}_n\right)_{n\geqslant1},\left(\mathbf{\xi}_{1,n}\right)_{n\geqslant1},\ldots, \left(\mathbf{\xi}_{m,n}\right)_{n\geqslant1},\left(\mathbf{\eta}_n\right)_{n\geqslant1}$ in $\mathbf{R}^N$ such that $\mathbf{p}_n,\mathbf{\xi}_{1,n},\cdots,\mathbf{\xi}_{m,n}$, $\mathbf{\eta}_n$ satisfy \textbf{E1,E3} and $\mathbf{p}_n\cdot\mathbf{\eta}_n>-1/n$ for each $n$. With $m+2$ applications of \textbf{WKL} we can construct an increasing sequence $\left(k_n\right)_{n\geqslant1}$ and points $\mathbf{p},\mathbf{\xi}_{1},\cdots,\mathbf{\xi}_{m},\mathbf{\eta}\in\mathbf{R}^N$ such that $\mathbf{p}_n\rightarrow\mathbf{p},\mathbf{\xi}_{i,n}\rightarrow\mathbf{\xi}_i$ $(1\leqslant i\leqslant m)$, and $\mathbf{\eta}_n\rightarrow\mathbf{\eta}$ as $n\rightarrow\infty$. The continuity of the demand functions, the dot product, and summation ensure that $\mathbf{p},\mathbf{\xi}_{1},\cdots,\mathbf{\xi}_{m},\mathbf{\eta}\in\mathbf{R}^N$ is a competitive equilibrium.

\bigskip
\noindent
Since Theorem \ref{cMcKenzie} allows for companies to make a loss (which can be made arbitrarily small), it is of interest to see how the results of this paper change if the producers profits are shared by the consumers. One should note that since rational entities will act continuously with respect to their profits, any change in consumer behaviour, resulting from companies being able to make a loss, can also be made arbitrary small. We leave a rigorous treatment of this set-up as an open problem.


\begin{thebibliography}{99}

\bibitem {AR}P.H.G. Aczel and M. Rathjen, \emph{Notes on Constructive Set
Theory}, Report No. 40, Institut Mittag-Leffler, Royal Swedish Academy of
Sciences, 2001.

\bibitem {BB}E.A. Bishop and D.S. Bridges, `Constructive Analysis',
Grundlehren der Math. Wiss., \textbf{279}, Springer-Verlag, Heidelberg, 1985.

\bibitem{dsb_82}D. Bridges, `Preference and utility: a constructive development', J. Math. Econom. 9, p. 165--185, 1982.

\bibitem{dsb_89}D. Bridges, `The constructive theory of preference relations on a locally compact space', Proc.
Koninklijke Nederlandse Akad. van Wetenschappen, Ser. A, 92(2), p. 141--165, 1989.

\bibitem{dsb_92}D. Bridges, `The construction of a continuous demand function for uniformly rotund preferences', Journal of Mathematical Economics, vol. 21,
p. 217--227, 1992.

\bibitem{dsb_93}D. Bridges, `Constructive notions of strict convexity', Math. Logic
Quarterly, vol. 39, p. 295--300, 1993.

\bibitem{dsb_94}D. Bridges, `The constructive theory of preference
relations on a locally compact space---II', Mathematical Social Sciences 21, p. l--9, 1994.

\bibitem {BR}D.S. Bridges and F. Richman, \emph{Varieties of Constructive
Mathematics}, London Math. Soc. Lecture Notes \textbf{97}, Cambridge Univ.
Press, 1987.

\bibitem{BRS} D.S. Bridges, F. Richman, and P. Schuster, `Linear
independence without choice', Ann. Pure Appl. Logic 101 (2000), p. 95--102.

\bibitem {BV}D.S. Bridges and L.S. V\^{\i}\c{t}\u{a}, \emph{Techniques of
Constructive Analysis}, Universitext, Springer-New-York, 2006.

\bibitem{mh}M. Hendtlass, `Kakutani's fixed point theorem in constructive mathematics', preprint.

\bibitem{mh2}M. Hendtlass, `Constructing the demand function of a strictly convex preference relation', preprint.

\bibitem{M1}L.W. McKenzie, `On the existence of general equilibrium for a competitive market', Econometrica 27, p. 54--71, 1959.

\bibitem{M2}L.W. McKenzie, `On the existence of general equilibrium for a competitive market: some corrections', Econometrica 29, p. 247--248, 1961.

\bibitem{M3}L.W. McKenzie, `The Classical Theorem on Existence of Competitive Equilibrium', Econometrica 49(4), p. 819--841, 1981.

\bibitem{Orehkov}V. P. Orevkov, `A constructive map of the square into itself, which moves every constructive point', Dokl. Akad. Nauk SSSR \textbf{152}, p. 55--58, 1963.

\bibitem{Scarf1}H. Scarf, `On the approximation of fixed points of continuous mappings', SIAM Journal of Applied Mathematics 15, p. 1328--1343, 1967.

\bibitem{HS}H. Schwichtenberg, `Realizability interpretation of proofs in constructive analysis', Theory Comput. Syst. 43(3-4), p. 583--602, 2008.

\bibitem{Tak}A. Takayama, \emph{Mathematical Economics}, The Dryden Press, Hinsdale, Illinois, 1974.

\bibitem{V}W. Veldman, `Brouwer's approximate fixed-point theorem is equivalent to Brouwer's fan theorem', \emph{Logicism, intuitionism, and formalism}, p. 277--299, Synth. Libr., 341, Springer, Dordrecht, 2009.

\end{thebibliography}
\end{document}